\documentclass[12pt]{amsart}
\usepackage{amssymb,latexsym}
\usepackage{enumerate}
\usepackage[margin=2in]{geometry}
\usepackage{graphicx}
\usepackage{psfrag}
\makeatletter
\@namedef{subjclassname@2010}{%
  \textup{2010} Mathematics Subject Classification}
\makeatother

\newtheorem{thm}{Theorem}
\newtheorem{cor}{Corollary}
\newtheorem{prop}{Proposition}

\theoremstyle{definition}

\newtheorem{rem}{Remark}
\newtheorem{exa}{Example}

\numberwithin{equation}{section}

\newcommand\A[1]{A_{#1}}
\newcommand\R{\mathbb{R}}
\newcommand\N{\mathbb{N}}

\newcommand{\sr}[1]{\mathfrak{M}_{#1}}
\newcommand{\B}[1]{B_{#1}}

\newcommand{\norma}[1]{\left\| #1 \right\| }
\newcommand{\abs}[1]{\left| #1 \right| }

\DeclareMathOperator{\sign}{sign}

\begin{document}


\baselineskip=17pt



\title{A new estimate of the difference among quasi-arithmetic means}

\author{Pawe{\l} Pasteczka}
\address{Institute of Mathematics\\University of Warsaw\\
02-097 Warszawa, Poland}
\email{ppasteczka@mimuw.edu.pl}

\date {October 27, 2013}

\begin{abstract}
In the 1960s Cargo and Shisha proved some majorizations for the distance among 
quasi-arithmetic means (defined as $f^{-1}\left( \sum_{i=1}^{n} w_i f(a_i)\right)$ 
for any continuous, strictly monotone function $f \colon I \rightarrow \mathbb{R}$, 
where $I$ is an interval, and $a=(a_1,\ldots,a_n)$ is a vector with entries in $I$, 
$w=(w_1,\ldots,w_n)$ is a sequence of corresponding weights $w_i>0$, $\sum w_i=1$). 

Nearly thirty years later, in 1991, P\`ales presented an iff condition 
for a sequence of quasi-arithmetic means to converge to another QA mean. It was 
closely related with the three parameters' operator $(f(x)-f(z))/(f(x)-f(y))$.

The author presented recently an estimate for the distance among such quasi-arithmetic 
means whose underlying functions satisfy some smoothness conditions. Used was the operator 
$f \mapsto f''/f'$ introduced in the 1940s by Mikusi\'nski and \L{}ojasiewicz. 
It is natural to look for similar estimate(s) in the case of the underlying 
functions \textit{not\,} being smooth. For instance, by the way of using 
P\`ales' operator. This is done in the present note. Moreover, the result 
strengthens author's earlier estimates. 

\end{abstract}

\subjclass[2010]{Primary 26E60; Secondary 26D15, 26D07}

\keywords{quasi-arithmetic means, metric}

\maketitle
\section{Introduction}
One of the most popular families of means encountered in literature is the family of 
\textit{quasi-arithmetic means}. Such a mean is defined for any continuous, strictly 
monotone function $f \colon U \rightarrow \mathbb{R}$, $U$ -- an interval. 
When $a = (a_1,\dots,a_n)$ is an arbitrary sequence of points in $U$
and $w=(w_1,\dots,w_n)$ is a sequence of corresponding \textit{weights}
($w_i>0$, $\sum w_i=1$) then the mean $\sr{f}(a,w)$ is defined by the equality
$$
\sr{f}(a,w) := f^{-1}\left( \sum_{i=1}^{n} w_i f(a_i)\right).
$$
This family was introduced in the beginning of the 1930s in a series of nearly simultaneous papers 
\cite{deFinetti,kolmogoroff,nagumo} as a generalization of the well-known \textit{power means}.

In the 1960s Cargo and Shisha \cite{CS64,CS69} introduced a metric among quasi-arithmetic means.
Namely, if $f$ and $g$ are both continuous, strictly monotone and have the same domain, then one can define 
a distance 
$$
\rho(\sr{f},\sr{g}):=\sup\{\abs{\sr{f}(a,w)-\sr{g}(a,w)} \colon a\textrm{ and }w \textrm{ admissible}\}.
$$
They also furnished some majorations for $\rho(\sr{f},\sr{g})$. One of their results is the proposition below;
hereafter $\norma{\cdot}_p$ denotes the standard $L^p$ norm ($1\le p \le \infty$).

In the present note, if not otherwise stated, the intervals are arbitrary.
\begin{prop}[\cite{CS69},Theorem~4.2]
\label{prop:CS69}
Let $U$ be an interval, $g \in \mathcal{C}(U)$ be strictly monotone, $f \in \mathcal{C}^1(U)$, $\inf \abs{f'} > 0$. 
Then $\rho(\sr{f},\sr{g})\le\frac{2\norma{f-g}_{\infty}}{\inf \abs{f'}}$.
\end{prop}

Using another result from \cite{CS69}, the author proved in \cite{P13} an alternative estimate for the distance between two quasi-arithmetic means satisfying some smoothness
conditions. An important tool was the operator $\A{}$, $\A{f}:=f''/f'$ introduced by Mikusi\'nski\footnote{
Mikusi\'nski and, independently, \L{}ojasiewicz, proved that comparability of quasi-arithmetic means
might be easily expressed in terms of operator $\A{}$. Besides, in mathematical economy, this operator used to be called the Arrow-Pratt index of absolute risk aversion.}
in \cite{mikusinski}. The crucial result was the following

\begin{prop}[\cite{P13}, Theorem~3]
\label{prop:rightright}
Let $U$ be a closed, bounded interval and $f,g \in C^{2}(U)$ have nonvanishing first derivatives. Then 
$$
\rho(\sr{f},\sr{g})
\le \abs{U} \exp(2\norma{A_f}_1) \sinh \left( 2 \norma{A_g-A_f}_1 \right).
$$
\end{prop}

\begin{rem}
\label{rem:sym}
Note that the left hand side is symmetric, while the right one is not. One could
symmetrize it using the $\min$ function. Nevertheless, this operation will be omitted to 
make notation more compact. This remark applies also to Theorem~\ref{thm:rightrightstr}.
\end{rem}

The proposition above has a noteworthy 
\begin{cor}[\cite{P13}, Corollary~3]
\label{cor:rightright}
Let $U$ be a closed, bounded interval and $f \in \mathcal{C}^{2}(U)$, $f' \ne 0$ everywhere. 
Moreover, let $(f_n)_{n \in \N}$ be a sequence of 
functions from $\mathcal{C}^{2}(U)$ with nonvanishing first derivatives, satisfying $\A{f_n} \rightarrow \A{f}$
in $L^1(U)$.
Then $\sr{f_n} \rightrightarrows \sr{f}$ uniformly with respect to $a$
and $w$.
\end{cor}

Note that the implication converse to that in Corollary~\ref{cor:rightright} does not hold.
This might be observed in the following

\begin{exa}
\label{exa:rightright}
Let $U=[0,2\pi]$, $f_n(x)=x+n^{-2} \sin (nx)$, $n \ge 2$ and $f(x)=x$ for $x \in U$. Then, by Proposition~\ref{prop:CS69}, $\rho(\sr{f},\sr{f_n})\le 2n^{-2}$. On the other hand it can be proved that $\norma{\A{f_n}-\A{f}}_1 = 2n\ln(n+1)\ge 4\ln 3$ for every $n \ge 2$.
\end{exa}

This drawback is implied by the fact that ,,the first norm does not see cancellations of integrals''. 
On the other hand, there was a couple of additional monotonicity assumptions in the very background of consideration in \cite{P13}. Namely, the mapping $n \mapsto \A{f_n}(x)$ was assumed to be either increasing for every $x$ or else decreasing for every $x$. Hence there was no point to care there about the problem above.

The situation looks fairly different in the present paper. As we will see, to deal with examples like the one above, 
it is more convenient to use another norm, $\norma{\cdot}_\ast$, which will be proposed in section~\ref{sec:strSS}.

Historically, some of results presented above (especially Corollary~\ref{cor:rightright})
correspond with an earlier result of P\'ales \cite{Pales}. Namely, using 
the three parameters' operator $\B{f}(x,y,z):=\frac{f(x)-f(y)}{f(x)-f(z)}$
defined on $\{(x,y,z) \in U^3\colon x \ne z\}=:\Delta$, he proved the following
\begin{prop}[\cite{Pales}, Corollary~1]
\label{prop:Pales}
Let $U$ be an interval, $f$ and $f_n$, $n \in \N$, be continuous, strictly monotone functions defined on $U$.

Then \Big( $\sr{f_n} \rightarrow \sr{f}$ pointwise \Big) $\iff$ \Big( $\B{f_n}\rightarrow \B{f}$ pointwise on $\Delta$ \Big).
\end{prop}

During the 15th International Conference on Functional Equations and Inequalities (2013), P\'ales himself asked about possible generalizations of the $(\Leftarrow)$ part of his theorem in the spirit of 
Proposition~\ref{prop:rightright} and Corollary~\ref{cor:rightright}. In other words, he asked how to majorize $\rho(\sr{f},\sr{g})$ in terms of $\B{}$?
On the other hand, it is natural to look for possible strengthening of Proposition~\ref{prop:rightright} 
and Corollary~\ref{cor:rightright} -- to eliminate the problem presented in Example~\ref{exa:rightright}.

In this paper we are going to propose such an estimate which not only implies the $(\Leftarrow)$ part of P\'ales' result 
but also leads to a strengthening of Proposition~\ref{prop:rightright}; cf. Corollary~\ref{cor:conv} and Theorem~\ref{thm:rightrightstr}, 
respectively.

\section{Main Result}

The main idea is to use the elementary fact that on compact sets pointwise convergence coincides 
with the uniform one. However, $\Delta$ is not compact (even if $U$ is). Therefore, finding a proper, 
compact subset of $\Delta$ seemed to be of utmost importance in the search for an estimate 
for the distance among means.

We observe that, when $x$ approaches $z$, the operator $\B{}$ becomes unbounded. 
So it is natural to consider those points of $\Delta$ for which $x$ and $z$ are 
separated one from another. For any $\alpha > 0$ define 
$$
\Delta_\alpha := \{(x,y,z) \in U^3 \colon \abs{x-z} \geq \alpha\} \subset \Delta\,.
$$

We are going to prove the following
\begin{thm}
\label{thm:main}
Let $U$ be an interval, $f$ and $g$ be two continuous, strictly monotone functions defined on $U$, 
and $\alpha > 0$. Then $\norma{\B{f}-\B{g}}_{\infty,\Delta_\alpha}<1$ implies $\rho(\sr{f},\sr{g})<\alpha$.
\end{thm}

It is useful to recall some basic properties of $\B{}$ before giving a proof. Namely, for any $f$,
\begin{align}
\B{f}(x,y,z)+\B{f}(z,y,x)&=1 \textrm{ for all } (x,y,z)\in \Delta, \label{eq:dual} \\
\sum w_i \B{f}\Big(\sr{f}(a,w),a_i,z\Big)&=0 \textrm{ for all }a,\,w\textrm{, and admissible
 }z. \label{eq:propB}
\end{align}

\begin{proof}[Proof of Theorem~\ref{thm:main}]
Fix any $a \in U^n$ with corresponding weights $w$. We will write shortly 
$$
F:=\sr{f}(a,w) \quad\textrm{ and }\quad G:=\sr{g}(a,w).
$$

It is sufficient to find such an $i\in\{1,\ldots,n\}$ that $(F,a_i,G) \notin \Delta_\alpha$.
Then, $\abs{F-G}<\alpha$, by the very definition of $\Delta_\alpha$.

Suppose conversely that $(F,a_i,G) \in \Delta_\alpha$ for all $i\in\{1,\ldots,n\}$.
In particular,
$$
\abs{\B{f}(F,a_i,G)-\B{g}(F,a_i,G)}<1 \textrm{ for all }i\in\{1,\ldots,n\}.
$$
Hence, by \eqref{eq:dual} and \eqref{eq:propB}, one obtains
\begin{align}
-1&<\sum_{k=1}^{n} w_k \Big(\B{f}(F,a_k,G)-\B{g}(F,a_k,G)\Big) \nonumber \\
&=\sum_{k=1}^{n} w_k \B{f}(F,a_k,G)+\sum_{k=1}^{n} w_k \Big(-1+\B{g}(G,a_k,F)\Big) \nonumber \\
&=-1+\sum_{k=1}^{n} w_k \B{g}(G,a_k,F) =-1. \nonumber
\end{align}
This contradiction ends the proof.
\end{proof}

\section{Applications}
\begin{cor}
\label{cor:conv}
Let $U$ be an interval, $f$ and $f_n$, $n \in \N$, be strictly monotone functions defined on $U$,
$\B{f_n} \rightarrow \B{f}$ pointwise. Then $\sr{f_n} \rightarrow \sr{f}$.

Moreover, if\, $U$ is compact then $\sr{f_n} \rightrightarrows \sr{f}$ with respect to $a$ and $w$.
\end{cor}
\begin{proof}
Fix any $a \in U^n$ with corresponding weights $w$, and a compact interval 
$K \subseteq U$ such that $a \in K^n$.

Then, for any $\alpha>0$, $\Delta_\alpha \cap K^3$ is compact. So $\B{f_n} \rightarrow \B{f}$ uniformly 
on $\Delta_\alpha \cap K^3$. Hence there exists $n_\alpha$ such that 
$$
\norma{\B{f_n}-\B{f}}_{\infty,\Delta_\alpha \cap K^3} <1 \textrm{ for all } n>n_\alpha.
$$
Hence, by Theorem~\ref{thm:main}, one obtains
$$
\rho(\sr{f_n}\vert_{K},\sr{f}\vert_{K})<\alpha \textrm{ for all } n>n_\alpha,
$$
where $\sr{f} \vert_{K}$ stands for the the means defined for vectors taking values in $K$.

To prove the moreover part one may simply take $K=U$.
\end{proof}

\begin{cor}
Let $U$ be a compact interval, $f$ and $f_n$, $n \in \N$, be a strictly monotone functions defined on $U$.
Then the following conditions are equivalent:
\begin{enumerate}[\upshape (i)]
\item $\B{f_n} \rightarrow \B{f}$ pointwise on $\Delta$,
\item $\sr{f_n} \rightarrow \sr{f}$ pointwise,
\item $\sr{f_n} \rightrightarrows \sr{f}$ with respect to $a$ and $w$.
\end{enumerate}
\end{cor}
Obviously \upshape{(iii)} $\Rightarrow$ \upshape{(ii)}. 
Moreover, by Proposition~\ref{prop:Pales}, \upshape{(i)} $\iff$ \upshape{(ii)}, 
while, by Corollary~\ref{cor:conv}, \upshape{(i)} $\Rightarrow$ \upshape{(iii)}. 

\begin{cor}
\label{cor:main}
If, in Theorem~\ref{thm:main}, the assumed inequality is not sharp, 
$\norma{\B{f}-\B{g}}_{\infty,\Delta_\alpha} \le 1$, then $\rho(\sr{f},\sr{g}) \le \alpha$.
\end{cor}

\subsection{\label{sec:strSS} Strengthening of Proposition~\ref{prop:rightright}}

Now we are going to propose some solution to the problem hinted at in Example~\ref{exa:rightright}.
Recalling, that problem arose from the fact that the closeness of functions does not imply closeness 
of their derivatives. Therefore, Proposition~\ref{prop:rightright} is completely useless in that 
example. Hence, in proposition's strengthening, instead of using the first norm, one needs 
to define some other norm omitting that drawback of the $L^1$ norm.
Let $U$ be an interval, $f \colon U \rightarrow \R$ be an arbitrary continuous function, 
and the `oscillation' norm be defined by 
$$
\norma{f}_\ast:=\sup_{a,b \in U}\abs{\int_{a}^{b} f(x)dx}.
$$
We are going to prove that Proposition~\ref{prop:rightright} might be 
strengthened to
\begin{thm}
\label{thm:rightrightstr}
Let $U$ be a closed, bounded interval and $f,g \in C^{2}(U)$ have nowhere nonvanishing 
first derivatives. Then 
$$
\rho(\sr{f},\sr{g}) \le \abs{U} \exp\norma{\A{f}}_\ast \Big(\exp \norma{\A{f}-\A{g}}_\ast -1\Big).
$$
\end{thm}

Note that this time the drawback discussed in the beginning of the section does not appear; 
cf. Example~\ref{exa:rightright2} later on. Moreover, $\norma{\cdot}_{\ast} \le \norma{\cdot}_1$, hence
the above theorem holds if one replaces $\norma{\cdot}_{\ast}$ by $ \norma{\cdot}_1$; Remark~\ref{rem:sym} is applicable here.

\begin{proof}[Proof of Theorem~\ref{thm:rightrightstr}]
Fix any $(x,y,z) \in \Delta$. We would like to majorize the value of $\abs{\B{f}(x,y,z)-\B{g}(x,y,z)}$.
Let us suppose, with no loss of generality, that
\begin{align}
f(s)&=\int_x^s \exp\left(\int_x^t \A{f}(u)du \right)dt, \nonumber \\
g(s)&=\int_x^s \exp\left(\int_x^t \A{g}(u)du \right)dt. \nonumber
\end{align}
Then 
\begin{align}
f(y)-f(x) &=\int_x^y \exp\left(\int_x^t \A{f}(u)du \right)dt \nonumber \\
&=\int_x^y \exp\left(\int_x^t \A{f}(u)-\A{g}(u)du\right) \exp\left(\int_x^t \A{g}(u)du \right)dt \nonumber \\
&=\int_x^y \exp\left(\int_x^t \A{f}(u)-\A{g}(u)du\right) g'(t) dt. \nonumber
\end{align}
By the mean value theorem, there exists $\xi \in I$ such that
\begin{align}
f(y)-f(x)  &= \exp\left(\int_x^\xi \A{f}(u)-\A{g}(u)du\right) \int_x^y  g'(t) dt \nonumber \\
&=\exp\left(\int_x^\xi \A{f}(u)-\A{g}(u)du\right) (g(y)-g(x)). \nonumber
\end{align}
Similarly,
$$
f(z)-f(x) = \exp\left(\int_x^\eta \A{f}(u)-\A{g}(u)du\right) (g(z)-g(x)) \textrm{ for some }\eta \in I.
$$
Therefore,
$$
\B{f}(x,y,z)= \exp\left(\int_\eta^\xi \A{f}(u)-\A{g}(u)du\right) \B{g}(x,y,z).
$$
So
$$
\exp\left(-\norma{\A{f}-\A{g}}_{\ast}\right) \abs{\B{g}(x,y,z)} \le \abs{\B{f}(x,y,z)}\le \exp\left(\norma{\A{f}-\A{g}}_{\ast}\right) \abs{\B{g}(x,y,z)}.
$$
But $\sign{\B{f}(x,y,z)}=\sign{\B{g}(x,y,z)}$ for any admissible $x$, $y$ and $z$. Hence one obtains
\begin{equation}
\abs{\B{f}(x,y,z)-\B{g}(x,y,z)} \le  \abs{\B{f}(x,y,z)} (\exp\norma{\A{f}-\A{g}}_{\ast}-1). \label{eq:1}
\end{equation}

Now we are going to majorize the value of $\abs{\B{f}(x,y,z)}$. But
$$
\abs{\B{f}(x,y,z)} = \abs{\frac{f(x)-f(y)}{f(x)-f(z)}} =  \abs{\frac{x-y}{x-z}} \frac{f'(p)}{f'(q)}\textrm{ for some }p,\,q \in U.
$$
Moreover $\abs{x-y}\le \abs{U}$ and
$$
\abs{\ln(\frac{f'(p)}{f'(q)})}=\abs{\int_{q}^{p} \A{f}(x)dx} \le \norma{\A{f}}_\ast.
$$
So
$$
\abs{\B{f}(x,y,z)} \le \frac{\abs{U}}{\abs{x-z}} \exp{\norma{\A{f}}_{\ast}}.
$$
Hence, continuing the inequality \eqref{eq:1},
$$
\abs{\B{f}(x,y,z)-\B{g}(x,y,z)} \le \frac{\abs{U}}{\abs{x-z}} \exp\norma{\A{f}}_\ast (\exp \norma{\A{f}-\A{g}}_{\ast}-1).
$$
Therefore, Corollary~\ref{cor:main} with proper $\alpha$ immediately gives
$$
\rho(\sr{f},\sr{g})\le \abs{U} \exp\norma{\A{f}}_\ast (\exp \norma{\A{f}-\A{g}}_{\ast}-1).
$$
\end{proof}

It might be proved that the right hand side of the above inequality can be majorized by the one appeared in Proposition~\ref{prop:rightright}. So this theorem is a strengthening of Proposition~\ref{prop:rightright}.
This theorem also has a corollary, which is a strengthening of Corollary~\ref{cor:rightright}, using $\norma{\cdot}_{\ast}$ instead of $L^1$, but it will be worded nowhere in this paper.

\begin{exa}
\label{exa:rightright2}
Let us take $U$, $f$ and $f_n$ like in Example~\ref{exa:rightright}.
Then $\A{f} \equiv 0$ so $\norma{\A{f}}_{\ast}=0$,
$$
\norma{\A{f_n}-\A{f}}_{\ast} = \sup_{a,b \in [0,2\pi]} \int_{a}^{b} \frac{- n·\sin nx}{n+ \cos nx} =\ln \left(\frac{n+1}{n-1} \right).
$$
So, by Theorem~\ref{thm:rightrightstr}, $\rho(\sr{f},\sr{f_n}) \le\tfrac{4\pi}{n-1}$.
This estimate is still much worse then one could expect (it is $\mathcal{O}(n^{-1})$ 
against $\mathcal{O}(n^{-2})$ ascertained in Example~\ref{exa:rightright}) but it is 
better then the one implied by Proposition~\ref{prop:rightright} ($\mathcal{O}(1)$).
\end{exa}


\begin{thebibliography}{99}
\bibitem{CS64} G.~T.~Cargo and O.~Shisha,
{\it On comparable means}, 
{Pacific J. Math.} {\bf 14} (1964), 1053--1058.

\bibitem{CS69} G.~T.~Cargo and O.~Shisha,
{\it A metric space connected with generalized means},
{J. Approx. Th.} {\bf 2} (1969), 207--222. 


\bibitem{deFinetti} B.~de Finetti,
{\it Sur concetto di media},
{Giornale dell' Istitutio Italiano degli Attuari} 
{\bf 2} (1931), 369--396.

\bibitem{kolmogoroff} A.~Kolmogoroff,
{\it Sur la notion de la moyenne},
{Rend. Accad. dei Lincei} {\bf 6} (1930), 388--391. 

\bibitem{mikusinski} J.~G.~Mikusi\'nski,
{\it Sur les moyennes de la forme $\psi^{-1}\left[\sum q \psi(x) \right]$}, 
{Studia Math.} {\bf 10} (1948), 90--96. 

\bibitem{nagumo} M.~Nagumo,
{\it Uber eine Klasse der Mittelwerte},
{Jap. Journ. of Math.} {\bf 7} (1930), 71--79.

\bibitem{Pales} Zs.~P\'ales,
{\it On the convergence of Means},
{J. Anal. Appl.} {\bf 156} (1991), 52--60.

\bibitem{P13} P. Pasteczka,
{\it When is a Family of Generalized Means a~Scale?},
Real Anal. Exchange {\bf 38} (2013), 193--210.


\end{thebibliography}
\end{document}